\documentclass[reqno]{amsproc}
\usepackage{amsthm,bbm}
\usepackage{amsfonts}
\usepackage{amsmath}
\usepackage{amssymb}
\usepackage{mathrsfs}
\usepackage{euscript}   
\usepackage{color}
\usepackage{mathbbol}
\usepackage{pifont}
\usepackage{tikz}

\numberwithin{equation}{section}

\newtheorem{thm}{Theorem}[section]
\newtheorem{pro}[thm]{Proposition}
\newtheorem{cor}[thm]{Corollary}
\newtheorem{lem}[thm]{Lemma}
\newtheorem{opq}[thm]{Problem}
\newtheorem*{opq*}{\bf Problem}

\theoremstyle{remark}
\newtheorem{rem}[thm]{Remark}

\theoremstyle{definition}
\newtheorem{exa}[thm]{Example}

\newcommand*{\cbb}{\mathbb{C}}
\newcommand*{\cov}[1]{\mathrm{cov}(#1)}

\newcommand*{\Ge}{\geqslant}
\newcommand*{\hh}{\EuScript{H}}

\newcommand*{\is}[2]{\langle#1,#2\rangle}
\newcommand*{\jd}[1]{\mathcal N(#1)}
\newcommand*{\kk}{\EuScript{K}}
\newcommand*{\Le}{\leqslant}
\newcommand*{\mcal}{\EuScript{M}}

\newcommand*{\ob}[1]{{\mathcal R}(#1)}
\newcommand*{\ogr}[1]{\boldsymbol B(#1)}

\begin{document}
   \title[Bishop-like theorems for non-subnormal operators]
{Bishop-like theorems for non-subnormal operators}
   \author[Z.\ J.\ Jab{\l}o\'nski]{Zenon Jan
Jab{\l}o\'nski}
   \address{Instytut Matematyki,
Uniwersytet Jagiello\'nski, ul.\ \L ojasiewicza 6,
PL-30348 Kra\-k\'ow, Poland}
\email{Zenon.Jablonski@im.uj.edu.pl}
   \author[I.\ B.\ Jung]{Il Bong Jung}
   \address{Department of Mathematics,
Kyungpook National University, Daegu 702-701, Korea}
   \email{ibjung@knu.ac.kr}
   \author[J.\ Stochel]{Jan Stochel}
\address{Instytut Matematyki, Uniwersytet
Jagiello\'nski, ul.\ \L ojasiewicza 6, PL-30348
Kra\-k\'ow, Poland} \email{Jan.Stochel@im.uj.edu.pl}
   \thanks{The research of the first and third
authors was supported by the National Science Center
(NCN) Grant OPUS No.\ DEC-2021/43/B/ST1/01651. The
research of the second author was supported by Basic
Science Research Program through the National Research
Foundation of Korea (NRF) funded by the Ministry of
Education (NRF-2021R111A1A01043569).}
   \begin{abstract}
The celebrated Bishop theorem states that an operator
is subnormal if and only if it is the strong limit of
a net (or a sequence) of normal operators. By the
Agler-Stankus theorem, $2$-isometries behave similarly
to subnormal operator in the sense that the role of
subnormal operators is played by $2$-isometries, while
the role of normal operators is played by Brownian
unitaries. In this paper we give Bishop-like theorems
for $2$-isometries. Two methods are involved, the
first of which goes back to Bishop's original idea and
the second refers to Conway and Hadwin's result of
general nature. We also investigate the strong and
$*$-strong closedness of the class of Brownian
unitaries.
   \end{abstract}
   \maketitle

   \section{Introduction}
We begin by providing the concepts and facts necessary
to understand the content of the article. In what
follows, $\cbb$ stands for the set of complex numbers.
Given two (complex) Hilbert spaces $\hh$ and $\kk,$ we
denote by $\ogr{\hh,\kk}$ the Banach space of all
bounded linear operators from $\hh$ to $\kk$. The
kernel and the range of an operator $T \in
\ogr{\hh,\kk}$ are denoted by $\jd{T}$ and $\ob T$,
respectively. We regard $\ogr{\hh}:=\ogr{\hh,\hh}$ as
a $C^*$-algebra. The identity operator on $\hh$ is
denoted by $I$. An operator $T\in \ogr{\hh}$ is said
to be {\em normal} if $T^*T=TT^*$, and {\em subnormal}
if there exists a Hilbert space $\kk$ and a normal
operator $N\in \ogr{\kk}$ such that $\hh \subseteq
\kk$ (isometric embedding) and $N$ is an extension of
$T$, that is, $Th=Nh$ for all $h\in \hh$. If $\hh$ is
infinite-dimensional, then the class of normal
operators on $\hh$ is strictly contained in the class
of subnormal operators on $\hh$ (for more information,
see \cite{Co91}). The celebrated Bishop theorem gives
the following topological characterization of
subnormal operators, where the ``strong closure''
refers to the ``strong operator topology'' on
$\ogr{\hh}$ (see Section~\ref{Sec.2}).
   \begin{thm}[{\cite[Theorem~3.3]{Bis57}; see also
\cite[Theorem~II.1.17]{Co91}}]
The strong closure of
the set of all normal operators on a Hilbert space
$\hh$ is equal to the set of all subnormal operators
on $\hh$.
   \end{thm}
In other words, an operator $T\in \ogr{\hh}$ is
subnormal if and only if there exists a net of normal
operators on $\hh$ strongly convergent to $T$.

There is another well-known class of operators called
$2$-isometries introduced by Agler in \cite{Ag90} and
intensively studied by Agler and Stankus in the
trilogy \cite{Ag-St95-6}. It resembles the class of
subnormal operators in the sense that each
$2$-isometry has an extension to a Brownian unitary.
In other words, $2$-isometries play the role of
subnormal operators, while Brownian unitaries play the
role of normal operators. Following \cite{Ag90}, we
say that an operator $T\in \ogr{\hh}$ is a {\em
$2$-isometry} if
   \begin{align*}
I-2T^*T+T^{*2}T^{2}=0.
   \end{align*}
Before defining Brownian unitary, recall that the {\em
covariance} of an operator $A\in \ogr{\hh}$ is defined
by $\cov{A}=\sqrt{\|A^*A-I\|}$. An operator $T$ is
said to be a {\em Brownian unitary of covariance}
$\sigma \in (0,\infty)$ if $T$ has a
matrix representation $T=\big[\begin{smallmatrix} V & \sigma  E \\
0 & U\end{smallmatrix}\big]$ relative to some
nontrivial orthogonal decomposition $\hh=\hh_1\oplus
\hh_2$, where $V\in \ogr{\hh_1}$ and
$E\in\ogr{\hh_2,\hh_1}$ are isometries such that
$\hh_1=\ob{V}\oplus \ob{E}$, and $U\in \ogr{\hh_2}$ is
unitary (see \cite[Proposition~5.12]{Ag-St95-6}).
Clearly, we have
   \begin{align*}
\text{$1 \Le \dim \hh_2 \Le \dim \hh_1 = \dim \hh$ and
$\dim \hh \Ge \aleph_0.$}
   \end{align*}
By a {\em Brownian unitary of covariance} $\sigma=0$,
we mean any unitary operator $T$. In both cases,
$\cov{T}=\sigma$. Apparently, the class of Brownian
unitaries is strictly contained in the class of
$2$-isometries. That the inclusion is strict can be
demonstrated with the help of a $2$-isometric
unilateral weighted shift $T$ with the weight sequence
$\{\sigma_n (\lambda)\}_{n=0}^{\infty}$, where
$\lambda \in (1,\infty)$ and
   \begin{align*}
\sigma_n (\lambda) = \sqrt{\frac{1+(n+1)
(\lambda^2-1)} {1+n (\lambda^2-1)}}, \quad n \Ge 0,
   \end{align*}
(see \cite[Lemma~6.1]{j-s01}), by showing that $T$
does not satisfy condition (iii) of
Theorem~\ref{chyrby}.

The aforementioned fundamental relationship between
$2$-isometries and Brownian unitaries is outlined
below (cf. Lemma~\ref{czw}).
   \begin{thm}[\mbox{\cite[Proposition~5.79 \& Theorem~5.80]{Ag-St95-6}}] \label{43}
An operator $T\in \ogr{\hh}$ is a $2$-isometry if and
only if $T$ extends to a Brownian unitary $R\in
\ogr{\kk}$, i.e., $\hh \subseteq \kk$ and $Th=Rh$ for
all $h\in \hh$. Moreover, for $\sigma \in [0,\infty)$,
any $2$-isometry of covariance less than or equal to
$\sigma$ extends to a Brownian unitary of covariance
$\sigma$.
   \end{thm}
The main goal of this paper is to prove a Bishop-like
theorem (or theorems) for $2$-isometries with Brownian
unitaries playing the role of normal operators. The
most general result of this kind, without dimensional
constraints, is as follows:
   \begin{thm} \label{liumyt}
If $T\in \ogr{\hh}$, then the following conditions are
equivalent{\em :}
   \begin{enumerate}
   \item[(i)] $T$ is a $2$-isometry,
   \item[(ii)] $T$ is the
strong limit of a net $\{T_{\tau}\}_{\tau\in
\varSigma} \subseteq \ogr{\hh}$ of Brownian unitaries
such that $\cov{T_{\tau}} = \cov{T}$ for every
$\tau\in \varSigma$,
   \item[(iii)] $T$ is the
strong limit of a net $\{T_{\tau}\}_{\tau\in
\varSigma} \subseteq \ogr{\hh}$ of Brownian unitaries
such that $\sup_{\tau\in \varSigma} \cov{T_{\tau}}
 < \infty$.
   \end{enumerate}
Moreover, for every $\sigma\in [0,\infty)$, the class
of $2$-isometries on $\hh$ of covariance less than or
equal to $\sigma$ is equal to the strong closure of
the class of Brownian unitaries on $\hh$ of covariance
less than or equal to $\sigma$.
   \end{thm}
For separable Hilbert spaces, the above topological
characterization of $2$-isome\-tries takes the
following form:
   \begin{thm} \label{stlymbu}
Suppose that $\hh$ is a separable Hilbert space and
$T\in \ogr{\hh}$. Then the following conditions are
equivalent{\em :}
   \begin{enumerate}
   \item[(i)] $T$ is a $2$-isometry,
   \item[(ii)]  $T$ is the
strong limit of a sequence $\{T_n\}_{n=1}^{\infty}
\subseteq \ogr{\hh}$ of Brownian unitaries such that
$\cov{T_n} = \cov{T}$ for every integer $n\Ge 1$,
   \item[(iii)] $T$ is the
strong limit of a sequence $\{T_n\}_{n=1}^{\infty}
\subseteq \ogr{\hh}$ of Brownian unitaries.
   \end{enumerate}
   \end{thm}
Using the Conway-Hadwin theorem (see
\cite[Theorem]{CoHad83}), we can obtain Bishop-like
theorems not only for strong topology, but also for
weak and $^*$-strong topology.
   \begin{thm} \label{cymprs}
Let $\hh$ be an infinite-dimensional separable Hilbert
space, $T\in \ogr{\hh}$ and $\sigma \in [0,\infty)$.
Then
   \begin{enumerate}
   \item[(i)] $T$ is the
strong limit of a sequence $\{T_n\}_{n=1}^{\infty}
\subseteq \ogr{\hh}$ of Brownian unitaries of
covariance $\sigma$ if and only if $T$ is a
$2$-isometry with $\cov{T} \Le \sigma$,
   \item[(ii)] $T$ is the
weak limit of a sequence $\{T_n\}_{n=1}^{\infty}
\subseteq \ogr{\hh}$ of Brownian unitaries of
covariance $\sigma$ if and only if $T$ is the
compression of a Brownian unitary of covariance
$\sigma$,
   \item[(iii)] $T$ is the $^*$-strong limit of a
sequence $\{T_n\}_{n=1}^{\infty} \subseteq \ogr{\hh}$
of Brownian unitaries of covariance $\sigma$ if and
only if $T$ is either a unitary operator or a Brownian
unitary of covariance $\sigma$.
   \end{enumerate}
   \end{thm}
Note that equivalence (i)$\Leftrightarrow$(ii) of
Theorem~\ref{stlymbu} follows from the statement (i)
of Theorem~\ref{cymprs} applied to $\sigma=\cov{T}$
and the fact that $2$-isometries on finite-dimensional
Hilbert spaces are unitary.

Regarding the Conway-Hadwin theorem, note that one of
its two main assumptions is that the considered
abstract class of operators is closed under countable
orthogonal sums. In our case, this is the class of
Brownian unitaries. Surprisingly, it is not closed
under this operation. However, the class of Brownian
unitaries of a fixed covariance is closed under
countable orthogonal sums. This narrower class is also
closed under approximate equivalence, which is the
second main assumption of the Conway-Hadwin theorem
(see Proposition~\ref{apyyqe}).
   \begin{pro} \label{orhsy}
The orthogonal sum $T=\bigoplus_{\omega \in \varOmega}
T_{\omega}$ of operators $T_{\omega}\in
\ogr{\hh_{\omega}}$ is a Brownian unitary if and only
if for every $\omega \in \varOmega$, the operator
$T_{\omega}$ is a Brownian unitary and
$\cov{T_{\omega}}$ equals either $0$ or $\cov{T}$.
   \end{pro}
It turns out that the class of Brownian unitaries
behaves quite well with respect to the $^*$-strong
topology.
   \begin{thm} \label{stulyn}
Let $\{T_n\}_{n=1}^{\infty}\subseteq \ogr{\hh}$ be a
sequence of Brownian unitaries which converges
$^*$-strongly to $T\in \ogr{\hh}$. Then $T$ is either
an isometry or a Brownian unitary. Moreover, if
$\cov{T_n}=\cov{T_1}$ for every $n\Ge 1$, then $T$ is
either a unitary operator or a Brownian unitary with
$\cov{T} = \cov{T_1}$.
   \end{thm}
Note that if $\hh$ is separable, then the ``only if''
part of statement (iii) of Theorem~\ref{cymprs} is
equivalent to the ``moreover'' part of
Theorem~\ref{stulyn}.

It is worth pointing out that Theorem~\ref{stulyn}, as
well as Corollary~\ref{unyoub}, are no longer valid if
the $^*$-strong topology is replaced by the strong
topology (see Remark~\ref{nuprth}). It may be the case
that the $^*$-strong limit of a sequence of Brownian
unitaries, each with positive covariance, is a
non-unitary isometry, meaning that the class of
Brownian unitaries is not $^*$-strongly closed. In
fact, it is not closed even in the operator norm
topology (see Example~\ref{sslnv}).

The question of the strong and weak closedness of the
class of $2$-isome\-tries has the following partial
answer.
   \begin{thm} \label{clidr}
The class of $2$-isometries on a Hilbert space $\hh$
is sequentially strongly closed, and if $\hh$ is
infinite-dimensional, it is not sequentially weakly
closed.
   \end{thm}
Taking into account the above, the following problem
arises.
   \begin{opq}
Is the class of $2$-isometries on an
infinite-dimensional Hilbert space strongly
closed\/{\em ?}
   \end{opq}
Let us recall that the strong closure of the set of
unitary (resp.\ normal) operators is equal to the set
of isometric (resp.\ subnormal) operators (see
\cite[Problem~225]{Hal82} and
\cite[Theorem~3.3]{Bis57}, respectively).
Consequently, the sets of isometric and subnormal
operators are strongly closed. It turns out that the
set of hyponormal operators is strongly closed as well
(see \cite[Problems~226]{Hal82}). The question of weak
closedness of these sets of operators is discussed in
Remark~\ref{ibiu}.

The paper is organized as follows. Section~\ref{Sec.2}
provides some basic properties of the covariance map,
including its strong lower semicontinuity (see
Lemma~\ref{cuw}). Sections~\ref{Sec.3} and \ref{Sec.4}
establish the proofs of the results stated in
Introduction. Section~\ref{Sec.5} gives two examples
showing that unitary operators can be $^*$-strong
limits of sequences of Brownian unitaries of a fixed
positive covariance.

In this paper, we focus on the study of
$2$-isometries. It is natural to ask what can be done
in the context of $m$-isometries where $m\Ge 3$. It
turns out that the situation is much more complex even
for $3$-isometries. The techniques used here for
$2$-isometries are no longer applicable for $m\Ge 3$.
The main reason for this is the lack of algebraic
characterizations of $m$-Brownian unitaries playing a
role somewhat similar to that played by Brownian
unitaries in the case of $2$-isometries (see
\cite{Cr-Su24,Su23}).
   \section{\label{Sec.2}The covariance map}
We begin by recalling the three basic operator
topologies that are explored in this paper. Denote by
$\ogr{\hh}$ the $C^*$-algebra of all bounded linear
operators on a complex Hilbert space $\hh$. The {\em
weak}, the {\em strong} and the {\em $^*$-strong}
topologies are defined as the locally convex
topologies on $\ogr{\hh}$ induced by the following
families of seminorms
   \begin{align*}
\ogr{\hh} \ni T & \longmapsto |\is{Tf}{g}|, \quad
f,g\in \hh,
   \\
\ogr{\hh} \ni T & \longmapsto \|Tf\|, \quad f\in \hh,
   \\
\ogr{\hh} \ni T & \longmapsto \sqrt{\|Tf\|^2 +
\|T^*f\|^2}, \quad f\in \hh.
   \end{align*}
For more information on the relationship between the
different operator topologies, we refer the reader to
\cite[Chapter~II]{Tak02}. In this paper, we will
repeatedly use the following fact, called the
sequential joint strong continuity of multiplication
of operators (see \cite[Lemma~3]{St87}; see also
\cite[Lemma~1]{CoHad83}).
   \begin{lem} \label{jcty}
If $\{T_n\}_{n=1}^{\infty}\subseteq \ogr{\hh}$ and
$\{R_n\}_{n=1}^{\infty}\subseteq \ogr{\hh}$ are
sequences converging strongly to $T\in \ogr{\hh}$ and
$R\in \ogr{\hh}$, respectively, then
$\{T_nR_n\}_{n=1}^{\infty}$ converges strongly to
$TR$.
   \end{lem}
Recall that the {\em covariance} of an operator $T\in
\ogr{\hh}$, denoted by $\cov{T}$, is defined by
$\cov{T}=\sqrt{\|\triangle_{T}\|}$, where
$\triangle_{T}=T^*T-I$. It is easy to see that
   \begin{align} \label{tles}
\|T\| & \Le \sqrt{1+\cov{T}^2},
   \\  \label{tlex}
\cov{T} & \Le \sqrt{1+\|T\|^2}.
   \end{align}
The proof of the following observation is left to the
reader.
   \begin{lem} \label{cysoof}
If $T=\bigoplus_{\omega \in \varOmega} T_{\omega}$ is
the orthogonal sum of operators $T_{\omega}\in
\ogr{\hh_{\omega}}$, then $\cov{T}=\sup_{\omega \in
\varOmega}\cov{T_{\omega}}$.
   \end{lem}
As shown below, the covariance map $\cov{\cdot}$ is
strongly lower semicontinuous. However, it is not
strongly continuous; worse, it is not $^*$-strongly
continuous (see Examples~\ref{prz1} and \ref{przew2}).
   \begin{lem} \label{cuw}
The map $\ogr{\hh}\ni T \mapsto \cov{T}\in [0,\infty)$
is strongly lower semicontinuous and norm continuous.
Moreover, if $\{T_{\tau}\}_{\tau\in \varSigma}
\subseteq \ogr{\hh}$ is a net converging strongly to
$T\in \ogr{\hh}$, then
   \begin{align*}
\cov{T} \Le \liminf_{\tau\in \varSigma}
\cov{T_{\tau}}.
   \end{align*}
   \end{lem}
   \begin{proof}
Since the norm continuity of $\cov{\cdot}$ is obvious,
we can focus on proving the strong lower
semicontinuity. Denote by $\mathbb{S}$ the unit sphere
of $\hh$ centered at $0$. For a vector $h\in
\mathbb{S}$, define the function $\varphi_h\colon
\ogr{\hh} \to [0,\infty)$ by
   \begin{align*}
\varphi_h(T)= \big|\|Th\|^2-\|h\|^2\big|, \quad T\in
\ogr{\hh}.
   \end{align*}
Clearly, each $\varphi_h$ being strongly continuous is
strongly lower semicontinuous. This implies that
$\varphi = \sup_{h\in \mathbb{S}} \varphi_h$ is
strongly lower semicontinuous. However,
   \begin{align*}
\varphi(T) = \sup_{h\in \mathbb{S}} |\is{\triangle_T
h}h| = r_{\mathrm{n}}(\triangle_T), \quad T\in
\ogr{\hh},
   \end{align*}
where $r_{\mathrm{n}}(S)$ stands for the numerical
radius of an operator $S$. Since the numerical radius
and the norm of a selfadjoint operator coincide (see
\cite[Theorem~1.4-2]{Gus-Rao97}), we conclude that
$\varphi(T)=\cov{T}$. Hence, the covariance map is
strongly lower semicontinuous.

The ``moreover'' part is a direct consequence of
\cite[Proposition~1.5.11]{Ped89}.
   \end{proof}
The covariance map $\cov{\cdot}$ is monotonically
increasing with respect to the inclusion of operators.
   \begin{lem} \label{czw}
If $\hh$ is a closed space invariant for
$R\in\ogr{\kk}$, then
   \begin{align*}
\cov{R|_{\hh}}\Le \cov{R}.
   \end{align*}
   \end{lem}
   \begin{proof}
Set $T=R|_{\hh}$. Noting that for every $h\in \hh$,
   \begin{align*}
|\is{\triangle_{T}h}{h}| = \big|\|Th\|^2 -
\|h\|^2\big| = \big|\|Rh\|^2 - \|h\|^2\big| =
\big|\is{\triangle_{R}h}{h}\big| \Le \cov{R}^2
\|h\|^2,
   \end{align*}
we deduce that (cf. the proof of Lemma~\ref{cuw})
   \begin{align*}
\cov{T}^2=\|\triangle_{T}\| =
r_{\mathrm{n}}(\triangle_{T}) \Le \cov{R}^2,
   \end{align*}
which completes the proof.
   \end{proof}
How the covariance map behaves when restricting
$2$-isometries to invariant or reducing subspaces can
be found in \cite[Lemma~5.32, Corollary~5.57 and
Lemma~5.90]{Ag-St95-6}.
   \section{\label{Sec.3}Proofs of Proposition~\ref{orhsy} and
Theorems~\ref{stulyn} and \ref{clidr}} Brownian
unitaries can be characterized completely
algebraically as follows.
   \begin{thm}[\mbox{\cite[Theorem~5.20]{Ag-St95-6}}]
\label{chyrby} Let $T \in \ogr{\hh}$ and
$\sigma=\cov{T}$. Suppose that $\sigma >0$. Then $T$
is a Brownian unitary if and only if the following
hold{\em :}
   \begin{enumerate}
   \item[(i)] $T^{*2}T^2-2T^*T+I=0$,
   \item[(ii)] $\triangle_{T} (TT^*-I)\triangle_{T}=0$,
   \item[(iii)] $\sigma^{-2}\triangle_{T}$ is an orthogonal projection,
   \item[(iv)] there exists an orthogonal projection $Q \in
\ogr{\hh}$ such that
   \begin{align*}
(\sigma^2-\triangle_{T})
(TT^*-I)(\sigma^2-\triangle_{T})=\sigma^4(\sigma^2-1)Q.
   \end{align*}
   \end{enumerate}
   \end{thm}
Below we will show that the class of all
$2$-isometries is sequentially strongly closed, but
not sequentially weakly closed if the underlying
Hilbert space is infinite-dimensional. For this we
need the following definition. An operator $T\in
\ogr{\hh}$ is said to be {\em weakly}, {\em strongly}
or {\em uniformly stable} if the power sequence
$\{T^n\}_{n=1}^{\infty}$ converges to $0$ weakly,
strongly or in norm, respectively. We refer the reader
to \cite{Eis10,Kub97} for more information on this
topic.
   \begin{proof}[Proof of Theorem~\ref{clidr}] That
the class of $2$-isometries on $\hh$ is sequentially
strongly closed follows easily from the definition and
Lemma~\ref{jcty}. It remains to prove that this class
is not sequentially weakly closed if $\dim \hh\Ge
\aleph_0$. Let $\hh$ be any infinite-dimensional
Hilbert space. Decompose it as $\hh=\hh_1\oplus\hh_2$,
where $\hh_1$ and $\hh_2$ are Hilbert spaces of the
same dimension. Fix a real number $q\in (0,1)$. Let
$V\in\ogr{\hh_1}$ be an isometry with $\dim \hh_1
\ominus \ob{V}=\dim \hh_1$ whose unitary part is
weakly stable (the case in which the unitary part is
absent is not excluded), let $E_0\in
\ogr{\hh_2,\hh_1}$ be an isometry with range
orthogonal to the range of $V$ (this is always
possible due to the assumption on dimensions of
$\hh_1$ and $\hh_2$), and let $X_0\in \ogr{\hh_2}$ be
an isometry. Set $T = \big[\begin{smallmatrix} V & E \\
0 & X \end{smallmatrix}\big] \in \ogr{\hh_1 \oplus
\hh_2}$ with $E = \sqrt{1-q}E_0$ and $X=\sqrt{q}X_0$.
Clearly, $X$ commutes with $E^*E$ and $X$ is
quasinormal, that is, $X$ commutes with $X^*X$. Since
$E^*E+X^*X-I=0$, we infer from
\cite[Theorem~7.1(iii)]{C-J-J-S21} that $T$ is a
$2$-isometry. Because powers of a $2$-isometry are
$2$-isometries (see \cite[Theorem~2.3]{Jab02}), we see
that $\{T^n\}_{n=1}^{\infty}$ is a sequence of
$2$-isometries. We claim that $\{T^n\}_{n=1}^{\infty}$
converges weakly to the zero operator. Indeed, it
follows from \cite[Lemma~3.3]{C-J-J-S23-S} that $T^n =
\big[\begin{smallmatrix} V^n & E_n \\
0 & X^n \end{smallmatrix}\big]$ for all integers $n\Ge
1$, where
   \begin{align} \label{fyrmen}
E_n = \sum_{j=0}^{n-1} V^jEX^{n-1-j}, \quad n\Ge 1,
   \end{align}
and
   \begin{align} \label{ynyn}
E_n^*E_n = E^*E \bigg(\sum_{j=0}^{n-1}X^{*j}X^j\bigg),
\quad n\Ge 1.
   \end{align}
This implies that
   \begin{align*}
\|E_n\|^2 = \|E_n^*E_n\| = (1-q) \sum_{j=0}^{n-1} q^j
= 1-q^n, \quad n\Ge 1.
   \end{align*}
Since $\{q^n\}_{n=1}^{\infty}$ converges to $0$, we
conclude that $\sup_{n\Ge 1} \|E_n\| < \infty$.
Because the unitary part of $V$ is weakly stable and,
by $\|X\|<1$, $X$ is uniformly stable, we deduce from
\cite[Corollary~2.2]{C-J-J-S23-S} that the sequence
$\{T^n\}_{n=1}^{\infty}$ of $2$-isometries converges
weakly to the zero operator which is obviously not a
$2$-isometry. This completes the proof.
   \end{proof}
   \begin{rem}
Regarding the proof of the second conclusion of
Theorem~\ref{clidr}, it is worth mentioning that a
Brownian unitary of positive covariance is never
weakly stable. Indeed, if $T= \big[\begin{smallmatrix}
V & \sigma E_0 \\ 0 & U
\end{smallmatrix}\big] \in \ogr{\hh_1 \oplus \hh_2}$
is a Brownian unitary of covariance $\sigma \in
(0,\infty)$, then $T^n =
\big[\begin{smallmatrix} V^n & E_n \\
0 & U^n \end{smallmatrix}\big]$ for all integers $n\Ge
1$, where $E_n$ is as in \eqref{fyrmen} with $E=\sigma
E_0$ and $X=U$. Hence, by \eqref{ynyn}, we have
   \begin{align*}
\|E_n\| = \sigma \sqrt{n}, \quad n \Ge 1,
   \end{align*}
which means that the sequence $\{E_n\}_{n=1}^{\infty}$
is not norm bounded. Therefore, by
\cite[Corollary~2.2]{C-J-J-S23-S}, $T$ is not weakly
stable. However, if $T$ is a Brownian unitary of
covariance $\sigma=0$ on an infinite-dimensional
Hilbert space, then it may happen that $T$ is weakly
stable (see \cite[Example~6.1]{C-J-J-S23-S}). In turn,
any $2$-isometry on a finite-dimensional Hilbert space
$\hh$, being automatically unitary, is never stable
(regardless of which Hausdorff vector space topology
is taken on $\ogr{\hh}$), and the class of
$2$-isometries on such $\hh$ is closed.
   \hfill $\diamondsuit$
   \end{rem}
   \begin{proof}[Proof of Theorem~\ref{stulyn}]
We begin by proving the first statement. There is no
loss of generality in assuming that
   \begin{align} \label{tiono}
\triangle_{T} \neq 0.
   \end{align}
By Theorem~\ref{chyrby}, we see that for all $n\Ge 1$,
   \begin{gather*}
T_n^{*2}T_n^{2} - 2T_n^*T_n + I = 0,
   \\
\triangle_{T_n}(T_nT_n^* - I)\triangle_{T_n} = 0.
   \end{gather*}
Using Lemma~\ref{jcty} and passing to the limit as
$n\to\infty$, we get
   \begin{gather*}
T^{*2}T^{2} - 2T^*T + I = 0,
   \\
\triangle_{T}(TT^* - I)\triangle_{T} = 0,
   \end{gather*}
which means that $T$ satisfies the conditions (i) and
(ii) of Theorem~\ref{chyrby}.

Set $\sigma_n=\cov{T_n}$ for $n\Ge 1$. If the set
$\{n\Ge 1\colon \sigma_n >0\}$ is finite, then passing
to a subsequence if necessary we can assume that each
$T_n$ is unitary. Clearly, the $^*$-strong limit of
unitaries is unitary, so $T$ is unitary, a
contradiction to \eqref{tiono}. Otherwise $\{n\Ge
1\colon \sigma_n
>0\}$ is infinite, so by passing to a subsequence if
necessary, we may assume that $\sigma_n > 0$ for all
$n\Ge 1$. It follows from the uniform boundedness
principle and \eqref{tlex} that the sequence
$\{\sigma_n\}_{n=1}^{\infty}$ is bounded. Using the
Heine-Borel theorem and passing to a subsequence if
necessary, we may assume that
$\{\sigma_n\}_{n=1}^{\infty}$ converges in
$[0,\infty)$, say to $\sigma$. Note that by
Theorem~\ref{chyrby}(iii), $\triangle_{T_n}^2 =
\sigma_n^2 \triangle_{T_n}$ for all $n\Ge 1$. Using
Lemma~\ref{jcty} and passing to the limit as
$n\to\infty$, it is seen that
   \begin{align} \label{tynsum}
\triangle_{T}^2 = \sigma^2 \triangle_{T}.
   \end{align}

Now, we consider two cases.

{\sc Case 1.} $\sigma > 0$.

Then, in view of \eqref{tynsum},
$P:=\sigma^{-2}\triangle_{T}$ is an orthogonal
projection. By \eqref{tiono}, we have $\|P\|=1$, or
equivalently $\cov{T}=\sigma$. This means that $T$
satisfies the condition (iii) of Theorem~\ref{chyrby}.

We claim that $T$ satisfies the condition (iv) of
Theorem~\ref{chyrby}. Set
   \begin{align*}
\nabla&=(\sigma^2-\triangle_{T})
(TT^*-I)(\sigma^2-\triangle_{T}),
   \\
\nabla_n&=(\sigma_n^2-\triangle_{T_n})
(T_nT_n^*-I)(\sigma_n^2-\triangle_{T_n}), \quad n\Ge
1.
   \end{align*}
Consider first the case where $\sigma \neq 1$. Passing
to a subsequence if necessary, we can assume that
$\sigma_n \neq 1$ for all $n\Ge 1$. It follows from
Theorem~\ref{chyrby}(iv) that
   \begin{align} \label{pryk}
\nabla_n^2= \sigma_n^4(\sigma_n^2-1) \nabla_n, \quad
n\Ge 1.
   \end{align}
Using Lemma~\ref{jcty} and passing to the limit as
$n\to \infty$, we see that $\nabla^2=
\sigma^4(\sigma^2-1)\nabla$, or equivalently that
$\nabla= \sigma^4(\sigma^2-1)Q$, where $Q\in
\ogr{\hh}$ is an orthogonal projection. Consider now
the case where $\sigma=1$. By
Theorem~\ref{chyrby}(iv), for every $n\Ge 1$, there
exists an orthogonal projection $Q_n\in \ogr{\hh}$
such that $\nabla_n = \sigma_n^4(\sigma_n^2-1)Q_n$.
Since $\lim_{n\to \infty}\|
\sigma_n^4(\sigma_n^2-1)Q_n\|=0$, we deduce from
Lemma~\ref{jcty} that $\nabla=0$, so $\nabla=
\sigma^4(\sigma^2-1)Q$ for any orthogonal projection
$Q\in \ogr{\hh}$. This proves our claim.

Summarizing the above and using Theorem~\ref{chyrby},
we see that $T$ is a Brownian unitary of covariance
$\sigma$.

{\sc Case 2.} $\sigma = 0$.

Then, by \eqref{tynsum}, $\triangle_{T}^2=0$. Since
$\triangle_{T}$ is selfadjoint, we conclude that
$\triangle_{T}=0$, which contradicts \eqref{tiono}.

Now we prove the ``moreover'' part. Set
$\sigma=\cov{T_1}$. It is easy to see that there is no
loss of generality in assuming that $\sigma > 0$.
Arguing as above and using the fact that
$\cov{T_n}=\sigma$ for every $n\Ge 1$, we deduce that
$P:=\sigma^{-2}\triangle_T$ is an orthogonal
projection. If $P\neq 0$, then, arguing again as
above, we see that $T$ is a Brownian unitary of
covariance $\sigma$. Suppose now that $P=0$. Then $T$
is an isometry. We claim that $T$ is unitary. As
above, passing to the limit in \eqref{pryk} as
$n\to\infty$, we deduce that $\nabla=
\sigma^4(\sigma^2-1)Q$, where $Q\in \ogr{\hh}$ is an
orthogonal projection. Hence, since $\triangle_T=0$,
we have
   \begin{align} \label{firpo}
P_{\jd{T^*}} = (1-\sigma^2)Q,
   \end{align}
where $P_{\jd{T^*}}$ stands for the orthogonal
projection of $\hh$ onto $\jd{T^*}$. Suppose contrary
to our claim that $T$ is not unitary, or equivalently
that $P_{\jd{T^*}}\neq 0$. Then by \eqref{firpo},
$1-\sigma^2=1$, and thus $\sigma=0$, a contradiction.
This completes the proof.
   \end{proof}
The following is a direct consequence of
Theorem~\ref{stulyn} and the fact that the class of
isometries (resp., unitaries) on a Hilbert space is
strongly (resp., $^*$-strongly) closed.
   \begin{cor} \label{unyoub}
The union of the classes of Brownian unitaries and
non-unitary isometries on $\hh$ is sequentially
$^*$-strongly closed. Moreover, for every $\sigma \in
[0,\infty)$, the union of the classes of Brownian
unitaries of covariance $\sigma$ and unitary operators
on $\hh$ is sequentially $^*$-strongly closed.
   \end{cor}
   \begin{rem} \label{nuprth}
It is worth noting that Theorem~\ref{stulyn} and
Corollary~\ref{unyoub} are no longer valid if the
$^*$-strong topology is replaced by the strong
topology. This can be seen by applying
Theorem~\ref{stlymbu} to any non-isometric
$2$-isometry $T$ on a separable Hilbert space $\hh$,
which is not a Brownian unitary. To have such a $T$,
consider an operator of
the form $T = \big[\begin{smallmatrix} V & \sigma E \\
0 & X \end{smallmatrix}\big]$, where $V$ and $E$ are
isometries such that $\ob{V}\perp \ob{E}$, $X$ is a
non-unitary isometry and $\sigma \in (0,\infty)$. It
is easy to see that $T$ is a non-isometric
$2$-isometry (see, e.g.,
\cite[Theorem~7.1]{C-J-J-S21}). That $T$ is not a
Brownian unitary follows from Theorem~\ref{chyrby} by
falsifying the condition (ii).
   \hfill $\diamondsuit$
   \end{rem}
On the other hand, as shown below, it may be the case
that the $^*$-strong limit of a sequence of Brownian
unitaries, each with positive covariance, is a
non-unitary isometry.
   \begin{exa} \label{sslnv}
Let $\hh$ be an infinite-dimensional Hilbert space.
Decompose it as $\hh=\hh_1\oplus \hh_2$ with $0<\dim
\hh_2 \Le \dim \hh_1$. Take a unitary operator $U\in
\ogr{\hh_2}$ and two isometries $V\in \ogr{\hh_1}$ and
$E\in \ogr{\hh_2,\hh_1}$ such that $\ob{V}\oplus
\ob{E} = \hh_1$ (this is always possible). Let
$\{\sigma_n\}_{n=1}^{\infty}$ be a sequence of
positive real numbers that converges to $0$. For $n\Ge
1$, set $T_n =
\big[\begin{smallmatrix} V & \sigma_n E \\
0 & U \end{smallmatrix}\big] \in \ogr{\hh_1 \oplus
\hh_2}$. Then each $T_n$ is a Brownian unitary of
covariance $\sigma_n$, and $\{T_n\}_{n=1}^{\infty}$
converges in norm to $V\oplus U$. This implies that
$\{T_n\}_{n=1}^{\infty}$ converges $^*$-strongly to
$V\oplus U$. Clearly, $V\oplus U$ being a non-unitary
isometry is not a Brownian unitary.
   \hfill $\diamondsuit$
   \end{exa}
As mentioned earlier, the question of whether the
class of Brownian unitaries is closed under orthogonal
sums has a negative answer (just consider the
orthogonal sum of Brownian unitaries of different
positive covariances; see also
Proposition~\ref{apyyqe}).
   \begin{proof}[Proof of Proposition~\ref{orhsy}]
Set $\sigma=\cov{T}$ and
$\sigma_{\omega}=\cov{T_{\omega}}$ for $\omega \in
\varOmega$. By Lemma~\ref{cysoof},
$\sigma=\sup_{\omega \in \varOmega} \sigma_{\omega}$.
Without loss of generality we can assume that $\sigma
>0$.

Consider first the case where $\sigma_{\omega} > 0$
for every $\sigma\in \varOmega$. Since
   \begin{align*}
\frac{\triangle_{T}}{\sigma^2} = \bigoplus_{\omega\in
\varOmega} \frac{\triangle_{T_\omega}}{\sigma^2},
   \end{align*}
we see that $\frac{\triangle_{T}}{\sigma^2}$ is an
orthogonal projection if and only if
$\frac{\triangle_{T_\omega}}{\sigma^2}$ is an
orthogonal projection for every $\omega \in
\varOmega$. Since $\triangle_{T_\omega}\neq 0$ for
every $\omega \in \varOmega$, each of the above two
equivalent conditions implies that
$\big\|\frac{\triangle_{T_\omega}}{\sigma^2}\big\|=1$,
or equivalently that $\sigma_\omega=\sigma$ for every
$\omega \in \varOmega$. This together with
Theorem~\ref{chyrby} applied to $T$ and $T_{\omega}$
implies that $T$ is a Brownian unitary if and only if
$T_{\omega}$ is a Brownian unitary and
$\sigma_{\omega}=\sigma$ for every $\omega\in
\varOmega$.

Suppose now that $\varOmega_1:=\{\omega\in
\varOmega\colon \sigma_{\omega} > 0\}$ is a proper
subset of $\varOmega$. Then the operator $T$
decomposes as $T= A \oplus B$, where
$A=\bigoplus_{\omega \in \varOmega_1} T_{\omega}$ and
$B=\bigoplus_{\omega \in \varOmega\setminus
\varOmega_1} T_{\omega}$. Clearly, by
Lemma~\ref{cysoof}, $\cov{A}= \sup_{\omega \in
\varOmega_1} \sigma_{\omega} = \sigma > 0$ and
$\cov{B}=0$. If $T$ is a Brownian unitary, then by
\cite[Lemma~5.32]{Ag-St95-6}, $A$ is a Brownian
unitary of covariance $\sigma > 0$ and $B$ is unitary.
Using what was done in the previous paragraph, we see
that $T_{\omega}$ is a Brownian unitary of covariance
$\sigma$ for every $\omega \in \varOmega_1$. Clearly,
$T_{\omega}$ is a Brownian unitary of covariance $0$
for every $\omega \in \varOmega \setminus
\varOmega_1$. In turn, if $T_{\omega}$ is a Brownian
unitary and $\sigma_{\omega} \in \{\sigma,0\}$ for
every $\omega \in \varOmega$, then, from what was done
in the previous paragraph, we deduce that $A$ is a
Brownian unitary of covariance $\sigma$. Moreover, $B$
is a Brownian unitary of covariance $0$. It is a
matter of routine to verify, by using
Theorem~\ref{chyrby}, that $T=A\oplus B$ is a Brownian
unitary. This completes the proof.
   \end{proof}
   \begin{cor} \label{wigrn}
Let $T=A \oplus B$ be the orthogonal sum of two
operators $A\in\ogr{\hh}$ and $B\in\ogr{\kk}$ for
which $\cov{A} > 0$ and $\cov{B}=0$. Then $T$ is a
Brownian unitary if and only if $A$ and $B$ are
Brownian unitaries.
   \end{cor}
   \begin{cor} \label{resyte}
For $\sigma \in [0,\infty)$, an operator $T\in
\ogr{\hh}$ is a Brownian unitary of covariance $0$ or
$\sigma$ if and only if $T$ is the restriction of a
Brownian unitary of covariance $\sigma$ to a reducing
subspace.
   \end{cor}
   \begin{proof}
For the ``only if'' part, note that a unitary operator
$T$ is the restriction of a Brownian unitary $T\oplus
B$ of covariance $\sigma$ (use Lemma~\ref{cysoof} and
Corollary~\ref{wigrn}) to a reducing subspace, where
$B$ is any Brownian unitary of covariance $\sigma$.
The ``if'' part is a direct consequence of
\cite[Lemma~5.32]{Ag-St95-6}.
   \end{proof}
Following \cite{CoHad83} (cf.\ \cite{Voi76,Had87}), we
say that a class (or rather a property) $\mathscr{E}$
of Hilbert space operators is {\em closed under
approximate equivalence} if for any operator $T\in
\ogr{\hh}$ of class $\mathscr{E}$ and any operator
$R\in \ogr{\kk}$ such that $\lim_{\tau\in \varSigma}
\|U_{\tau}^* T U_{\tau} - R\|=0$ for some net
$\{U_{\tau}\}_{\tau\in \varSigma}$ of unitary
operators $U_{\tau}\in \ogr{\kk,\hh}$, the operator
$R$ is of class $\mathscr{E}$. The class $\mathscr{E}$
is said to be {\em closed under unitary equivalence}
if for any operator $T\in \ogr{\hh}$ of class
$\mathscr{E}$ and any unitary operator $U\in
\ogr{\kk,\hh}$, the operator $U^*TU$ is of class
$\mathscr{E}$. Finally, $\mathscr{E}$ is called {\em
closed under orthogonal sums} if for any uniformly
bounded family $\{T_{\omega}\}_{\omega\in \varOmega}$
of operators of class $\mathscr{E}$, the orthogonal
sum $T=\bigoplus_{\omega \in \varOmega} T_{\omega}$ is
an operator of class $\mathscr{E}$. Clearly, the
closedness under approximate equivalence implies the
closedness under unitary equivalence.

As shown below, the class of Brownian unitaries of
fixed covariance is closed under each operation
mentioned above.
   \begin{pro}\label{apyyqe}
If $\sigma\in [0,\infty)$ is fixed, then the class of
Brownian unitaries of covariance $\sigma$ is closed
under unitary equivalence, approximate equivalence and
orthogonal sums.
   \end{pro}
   \begin{proof}
First, we show that the class of Brownian unitaries of
covariance $\sigma$ is closed under approximate
equivalence. Let $T\in \ogr{\hh}$ be a Brownian
unitary of covariance $\sigma$ and $R\in \ogr{\kk}$ be
such that $\lim_{\tau\in \varSigma} \|U_{\tau}^* T
U_{\tau} - R\|=0$ for some net of unitary operators
$U_{\tau}\in \ogr{\kk,\hh}$. There is no loss of
generality in assuming that $\sigma>0$. Applying the
inner automorphism $\ogr{\hh}\ni X \mapsto
U_{\tau}^*XU_{\tau}\in \ogr{\kk}$ to both sides of
each equality appearing in the conditions (i)-(iv) of
Theorem~\ref{chyrby}, and then passing to the limit as
$\tau\to\infty$, we deduce that the limit operator $R$
satisfies the conditions (i)-(iv) with
$\sigma=\cov{T}$. Since
   \begin{align*}
\cov{R}=\lim_{\tau\in \varSigma}
\cov{U_{\tau}^*TU_{\tau}}=\sigma,
   \end{align*}
we can apply Theorem~\ref{chyrby} to deduce that $R$
is a Brownian unitary of covariance $\sigma$.

The above implies that the class of Brownian unitaries
of covariance $\sigma$ is closed under unitary
equivalence. That this class is closed under
orthogonal sums follows immediately from
Proposition~\ref{orhsy} and Lemma~\ref{cysoof}.
   \end{proof}
   \section{\label{Sec.4}Proofs of Theorems~\ref{liumyt},
\ref{stlymbu} and \ref{cymprs}} We begin with the case
of arbitrary (not necessarily separable) Hilbert
spaces. The following result is inspired by
\cite[Theorem~3.3]{Bis57} and \cite[Theorem]{CoHad83}.
   \begin{lem} \label{jesd}
Let $\mathscr{E}$ be a class of Hilbert space
operators closed under unitary equivalence. Suppose
that $\hh$ and $\kk$ are Hilbert spaces of the same
orthogonal dimension such that $\hh \subseteq \kk$.
Let $R \in \ogr{\kk}$ be an operator of class
$\mathscr{E}$ and $T\in \ogr{\hh}$ be an operator such
that $T\subseteq R$. Then there exists a net
$\{T_{\tau}\}_{\tau\in \varSigma}\subseteq \ogr{\hh}$
of operators of class $\mathscr{E}$ converging
strongly to $T$ such that
$\|p(T_{\tau},T_{\tau}^*)\|=\|p(R,R^*)\|$ for every
complex polynomial $p$ in two noncommuting variables.
In particular, $\|T_{\tau}\|=\|R\|$ and
$\cov{T_{\tau}}=\cov{R}$ for every $\tau\in
\varSigma$. Moreover, if $\hh$ is separable, then the
net $\{T_{\tau}\}_{\tau\in \varSigma}$ can be replaced
by a sequence $\{T_n\}_{n=1}^{\infty}$ with the same
properties.
   \end{lem}
   \begin{proof}
There is no loss of generality in assuming that
$\dim\hh=\dim \kk \Ge \aleph_0$. Let $\varSigma$ be
the set of all nonzero finite-dimensional subspaces of
$\hh$ partially ordered by inclusion. Then for every
$\tau \in \varSigma$, there exists a unitary
isomorphism $U_{\tau} \in \ogr{\hh,\kk}$ such that
$U_{\tau}h=h$ for every $h\in \tau + T(\tau)$. Set
$T_{\tau}=U_{\tau}^{-1}RU_{\tau}$ for $\tau\in
\varSigma$. By assumption, $T_{\tau}$ is of class
$\mathscr{E}$ for every $\tau \in \varSigma$, and
$\{T_{\tau}\}_{\tau\in \varSigma}$ converges strongly
to $T$. Since for every complex polynomial $p$ in two
noncommuting variables,
   \begin{align*}
p(T_{\tau},T_{\tau}^*) = U_{\tau}^{-1} p(R,R^*)
U_{\tau},
   \end{align*}
we see that $\|p(T_{\tau},T_{\tau}^*)\|=\|p(R,R^*)\|$.

Now, we discuss the ``moreover'' part. Assume that
$\dim \hh=\aleph_0$. Then, fixing an orthonormal basis
$\{e_j\}_{j=1}^{\infty}$ of $\hh$, one can define
$T_n=T_{\tau_n}$ with $\tau_n=\bigvee\{e_j\}_{j=1}^n$
for $n\Ge 1$.
   \end{proof}
   \begin{proof}[Proof of Theorem~\ref{liumyt}]
(i)$\Rightarrow$(ii) Suppose that $T$ is a
$2$-isometry. According to
\cite[Theorem~5.80]{Ag-St95-6} and its proof, $T$
extends to a Brownian unitary $R\in \ogr{\kk}$ such
that $\cov{R}=\cov{T}$ and $\dim \hh=\dim \kk$. By
Proposition~\ref{apyyqe}, the class $\mathscr{E}$ of
Brownian unitaries $S$ with $\cov{S}=\cov{T}$
satisfies the assumptions of Lemma~\ref{jesd}.
Applying this lemma yields (ii).

The implication (ii)$\Rightarrow$(iii) is trivial.

(iii)$\Rightarrow$(i) Let $T$ be the strong limit of a
net $\{T_{\tau}\}_{\tau \in \varSigma}$ of Brownian
unitaries such that $\sup_{\tau\in \varSigma}
\cov{T_{\tau}} < \infty$. It follows from \eqref{tles}
that $\sup_{\tau\in \varSigma}\|T_{\tau}\| < \infty$.
By \cite[Lemma~1]{CoHad83}, we have
   \begin{align} \label{piweors}
\text{$\textrm{s-}\lim\nolimits_{\tau\in \varSigma}
T_{\tau}^k=T^k$ for every integer $k\Ge 0$.}
   \end{align}
Since any Brownian unitary is a $2$-isometry (see
Theorem~\ref{chyrby}(i)), we see that
   \begin{align*}
\|h\|^2 - 2 \|T_{\tau}h\|^2 + \|T_{\tau}^2h\|^2 =0,
\quad h\in \hh, \, \tau \in \varSigma.
   \end{align*}
Letting $\tau\to \infty$ and using \eqref{piweors}, we
conclude that
   \begin{align*}
\|h\|^2 - 2 \|Th\|^2 + \|T^2h\|^2 =0, \quad h\in \hh.
   \end{align*}
This means that $T$ is a $2$-isometry.

Now we prove the ``moreover'' part. The inclusion
``$\subseteq$'' is a direct consequence of the
implication (i)$\Rightarrow$(ii). The converse
inclusion follows from the implication
(iii)$\Rightarrow$(i) and Lemma~\ref{cuw}.
   \end{proof}
   \begin{proof}[Proof of Theorem~\ref{stlymbu}]
A careful inspection of the proof of the implication
(i)$\Rightarrow$(ii) of Theorem~\ref{liumyt} together
with the ``moreover'' part of Lemma~\ref{jesd} shows
that (i) implies (ii). The implication
(ii)$\Rightarrow$(iii) is trivial. To prove that (iii)
implies (i), suppose that $T$ is the strong limit of
$\{T_n\}_{n=1}^{\infty}$, where
$\{T_n\}_{n=1}^{\infty}$ is a sequence of Brownian
unitaries. Using the uniform boundedness principle and
\eqref{tlex}, we deduce that there exists $\sigma \in
[0,\infty)$ such that $\cov{T_n} \Le \sigma$ for all
integers $n\Ge 1$. Applying the implication
(iii)$\Rightarrow$(i) of Theorem~\ref{liumyt}, we
conclude that $T$ is a $2$-isometry.
   \end{proof}
Our next goal is to investigate the cases of strong,
weak and $^*$-strong topologies. Recall that if $\hh$
and $\kk$ are Hilbert spaces such that $\hh \subseteq
\kk$, then an operator $S\in \ogr{\hh}$ is called the
{\em compression} of an operator $R \in \ogr{\kk}$ to
$\hh$ if
   \begin{align} \label{nytinp}
S=P_{\hh}R|_{\hh},
   \end{align}
where $P_{\hh}\in \ogr{\kk}$ is the orthogonal
projection of $\kk$ onto $\hh$. We refer the reader to
\cite{Szym92}, where topological properties of the
compression map $R \mapsto P_{\hh}R|_{\hh}$ are
studied. We now state a version of Conway and Hadwin's
result with a correction on infinite-dimensionality
(see Remarks~\ref{ibiu} and \ref{tykire}).
   \begin{thm}[\mbox{\cite[Theorem]{CoHad83}}] \label{cohyd}
Suppose that $\hh$ is an infinite-dimensional
separable Hilbert space, and $\mathscr{E}$ is a class
of operators closed under approximate equivalence and
orthogonal sums. If $S\in \ogr{\hh}$, then{\em :}
   \begin{enumerate}
   \item[(i)] $S$ is the
strong limit of a sequence $\{S_n\}_{n=1}^{\infty}
\subseteq \ogr{\hh}$ of operators of class
$\mathscr{E}$ if and only if $S$ is the restriction of
an operator of class $\mathscr{E}$ to its invariant
subspace $\hh$,
   \item[(ii)] $S$ is the
weak limit of a sequence $\{S_n\}_{n=1}^{\infty}
\subseteq \ogr{\hh}$ of operators of class
$\mathscr{E}$ if and only if $S$ is the compression of
an operator of class $\mathscr{E}$ to $\hh$,
   \item[(iii)] $S$ is the $^*$-strong limit of a
sequence $\{S_n\}_{n=1}^{\infty} \subseteq \ogr{\hh}$
of operators of class $\mathscr{E}$ if and only if $S$
is the restriction of an operator of class
$\mathscr{E}$ to its reducing subspace $\hh$.
   \end{enumerate}
   \end{thm}
   \begin{rem}  \label{tykire}
The above version of Conway-Hadwin theorem differs
from the original in that the term ``unitary
equivalence'' does not appear in it. Formally
speaking, we have replaced ``isometric embedding'' by
``coordinate embedding''. To explain this, we argue as
follows. Suppose that $S\in \ogr{\hh}$ is unitarily
equivalent to $\tilde{S} \in \ogr{\tilde{\hh}}$, i.e.,
$\tilde{S}=U^*SU$ for some unitary operator $U\in
\ogr{\tilde \hh,\hh}$, and $\tilde{S} =
P_{\tilde{\hh}} \tilde{R} |_{\tilde{\hh}}$, where
$\tilde{R}\in \ogr{\tilde{\kk}}$ (cf.\
\eqref{nytinp}). Then taking a Hilbert space
$\EuScript{M}$ with $\dim \EuScript{M}= \dim
\tilde{\hh}^{\perp}$ and a unitary $W\in
\ogr{\tilde{\hh}^{\perp}, \EuScript{M}}$, we can
define the ``coordinate embedding'' $V\colon \hh \ni h
\mapsto (h,0)\in \hat{\kk}$ of $\hh$ into
$\hat{\kk}:=\hh\oplus \EuScript{M}$. Set $\hat{S} :=
\hat{V}S \hat{V}^* \in \ogr{\hat{\hh}}$ with
$\hat{V}\colon \hh \to \hat{\hh}:=\ob{V}$ given by
$\hat{V}h=Vh$ for $h\in \hh$, and $\hat{R}:=\hat{U}
\tilde{R} \hat{U}^* \in \ogr{\hat{\kk}}$ with $\hat{U}
\in \ogr{\tilde{\kk}, \hat{\kk}}$ given by $\hat{U} :=
U \oplus W$. Using the identity $P_{\tilde{\hh}} =
\hat{U}^* P_{\hat{\hh}} \hat{U}$, we conclude that
$\hat{S}= P_{\hat{\hh}} \hat{R} |_{\hat{\hh}}$, and
$\tilde{\hh}$ is invariant (resp., reducing) for
$\tilde{R}$ if and only if $\hat{\hh}$ is invariant
(resp., reducing) for $\hat{R}$. Clearly, $\tilde{R}$
is of class $\mathscr{E}$ if and only if $\hat{R}$ is
of class $\mathscr{E}$ (provided $\mathscr{E}$ is as
in Theorem~\ref{cohyd}). \hfill $\diamondsuit$
   \end{rem}
   \begin{rem} \label{ibiu}
Note that if $\hh$ is finite-dimensional, then
assertion (ii) of Theorem~\ref{cohyd} is not true in
general. For example, it is not true for the classes
$\mathscr{U}$ and $\mathscr{N}$ of unitary and normal
operators. Indeed, if it were true for
$\mathscr{E}=\mathscr{U}$, then since $\mathscr{U}$ is
closed in operator norm topology, the compression
$P_{\hh}U|_{\hh}$ of any unitary operator $U\in
\ogr{\kk}$ to any finite-dimensional subspace $\hh$ of
$\kk$ would be unitary, so any unitary operator would
be reduced by any finite-dimensional subspace, which
is not the case if $\dim \hh \Ge 2$. In turn, if it
were true for $\mathscr{E}=\mathscr{N}$, then since
$\mathscr{N}$ is closed in operator norm topology, the
compression of any normal operator $N\in\ogr{\kk}$ to
any finite-dimensional subspace $\hh$ of $\kk$ would
be normal. However, this is not the case. Indeed, the
operator $B=[\begin{smallmatrix} 0& 0 \\ 1 & 0
\end{smallmatrix}]$ is a non-normal compression of the unitary
operator $A=\Big[\begin{smallmatrix} 0& 0 & 1 \\
1 & 0 & 0 \\ 0 & 1 & 0 \end{smallmatrix}\Big]$ to
$\cbb^2 \oplus \{0\}$. This example also shows that
assertion (ii) of Theorem~\ref{cohyd} is not true for
$\mathcal{E}=\mathcal{U}$ if $2 \Le \dim \hh <
\aleph_0$. We refer the reader to \cite{Hol13}, where
the question of when compressions of normal matrices
are normal is investigated.

It is worth pointing out that by the Halmos's dilation
theorem (see \cite[Problem~222 and Corollary]{Hal82}),
the set of compressions of unitary (resp., normal)
operators to a fixed Hilbert space $\hh$ is equal to
the set of contractions (resp., operators) on $\hh$.
On the other hand, if $\dim \hh \Ge \aleph_0$, then
the weak closure of the set $\mathbb{U}_{\hh}$ (resp.,
$\mathbb{N}_{\hh}$) of unitary (resp., normal)
operators on $\hh$ is equal to the set of contractions
(resp., operators) on $\hh$ (see
\cite[Problem~224]{Hal82}). As a consequence, if $\dim
\hh \Ge \aleph_0$, then the weak closure of
$\mathbb{U}_{\hh}$ (resp., $\mathbb{N}_{\hh}$) is
equal to the set of compressions of unitary (resp.,
normal) operators to $\hh$. However, if $2 \Le \dim
\hh < \aleph_0$, this equality is false.
   \hfill $\diamondsuit$
   \end{rem}
   \begin{proof}[Proof of Theorem~\ref{cymprs}]
(i) By Proposition~\ref{apyyqe}, the class
$\mathscr{E}_{\sigma}$ of all Brownian unitaries of
covariance $\sigma$ satisfies the assumptions of
Theorem~\ref{cohyd}. Hence, by the statement (i) of
this theorem, the operator $T\in \ogr{\hh}$ is the
strong limit of a sequence $\{T_n\}_{n=1}^{\infty}
\subseteq \ogr{\hh}$ of Brownian unitaries of
covariance $\sigma$ if and only if $T$ is the
restriction of a Brownian unitary of covariance
$\sigma$ to its invariant subspace $\hh$. In turn, by
Lemma~\ref{czw} and Theorem~\ref{43}, the latter holds
if and only if $T$ is a $2$-isometry with $\cov{T} \Le
\sigma$.

(ii) This is a direct consequence of the statement
(ii) of Theorem~\ref{cohyd} applied to
$\mathscr{E}=\mathscr{E}_{\sigma}$.

(iii) Applying the statement (iii) of
Theorem~\ref{cohyd} to the class
$\mathscr{E}=\mathscr{E}_{\sigma}$, we see that $T$ is
the $^*$-strong limit of a sequence
$\{T_n\}_{n=1}^{\infty} \subseteq \ogr{\hh}$ of
Brownian unitaries of covariance $\sigma$ if and only
if $T$ is the restriction of a Brownian unitary of
covariance $\sigma$ to its reducing subspace $\hh$. By
Corollary~\ref{resyte}, the latter holds if and only
if $T$ itself is a Brownian unitary of covariance $0$
or $\sigma$. This completes the proof.
   \end{proof}
The following two facts are direct consequences of
Theorem~\ref{cymprs}.
   \begin{cor}[cf.\ Remark~\ref{ibiu}]
If $\hh$ is an infinite-dimen\-sion\-al separable
Hilbert space and $T\in \ogr{\hh}$, then $T$ is the
weak limit of a sequence $\{T_n\}_{n=1}^{\infty}
\subseteq \ogr{\hh}$ of unitary operators if and only
if $\|T\| \Le 1$.
   \end{cor}
   \begin{cor}
For $\sigma\in [0,\infty)$, any unitary operator on an
infinite-dimen\-sion\-al separable Hilbert space $\hh$
is a $^*$-strong limit of a sequence of Brownian
unitaries of covariance $\sigma$.
   \end{cor}
\section{\label{Sec.5}Examples}
In this section we provide some explicit examples of
unitary operators that are $^*$-strong limits of
sequences of Brownian unitaries of constant positive
covariance. In particular, this shows that the
covariance function $T \mapsto \cov{T}$ is not
$^*$-strongly continuous.
   \begin{exa} \label{prz1}
Let $\hh$ be an infinite-dimensional separable Hilbert
space and let $\{e_n\}_{n=1}^{\infty}$ be an
orthonormal basis of $\hh$. Fix a real number $\sigma
> 0$ and two sequences $\{\lambda_j\}_{j=1}^{\infty}$ and
$\{z_j\}_{j=1}^{\infty}$ of complex numbers of modulus
$1$. Set $\hh_{2,n}=\langle e_{n+1}\rangle$ and
$\hh_{1,n}=\hh_{2,n}^{\perp}$, where $\langle f
\rangle=\cbb \cdot f$ for $f\in \hh$. Clearly
$\hh=\hh_{1,n} \oplus \hh_{2,n}$ and
$\hh_{1,n}=\bigoplus_{j\neq n+1} \langle e_j \rangle$.
Define the unitary operator $W_n\in
\ogr{\bigoplus_{j=1}^n \langle e_j \rangle}$ by
   \begin{align*}
W_n\Big(\sum_{j=1}^n \xi_j e_j\Big) = \sum_{j=1}^n
\lambda_j \xi_j e_j, \quad (\xi_1, \ldots, \xi_n)\in
\cbb^n,
   \end{align*}
and the unilateral shift $\widetilde{V}_n\in
\ogr{\bigoplus_{j=n+2}^{\infty} \langle e_j \rangle}$
of multiplicity $1$ by
   \begin{align*}
\widetilde{V}_n\Big(\sum_{j=n+2}^{\infty} \xi_j
e_j\Big) = \sum_{j=n+2}^{\infty} \xi_j e_{j+1}, \quad
\{\xi_{j+n+1}\}_{j=1}^{\infty} \in \ell^2.
   \end{align*}
Set $V_n= W_n \oplus \widetilde{V}_n \in
\ogr{\hh_{1,n}}$. Define the operators $E_n\in
\ogr{\hh_{2,n},\hh_{1,n}}$ and $U_n\in
\ogr{\hh_{2,n}}$ by $E_n(\xi e_{n+1})=\xi e_{n+2}$ and
$U_n(\xi e_{n+1}) = z_n \xi e_{n+1}$ for $\xi\in
\cbb$. Clearly, $V_n$ and $E_n$ are isometries such
that $\hh_{1,n}=\ob{V_n}\oplus \ob{E_n}$. Define the
Brownian unitary $T_n\in \ogr{\hh}$ by $T_n=\big[\begin{smallmatrix} V_n & \sigma E_n \\
0 & U_n \end{smallmatrix}\big]$ relative to
$\hh=\hh_{1,n} \oplus \hh_{2,n}$. The covariance of
each $T_n$ is equal to $\sigma$. Finally, let $T\in
\ogr{\hh}$ be the unitary operator defined by
   \begin{align*}
Th=\sum_{j=1}^{\infty} \lambda_j \is{h}{e_j} e_j,
\quad h \in \hh.
   \end{align*}

We claim that the sequence $\{T_n\}_{n=1}^{\infty}$
converges $^*$-strongly to $T$. First, we show that
$\{T_n\}_{n=1}^{\infty}$ converges strongly to $T$.
For, take a vector $h \in \hh$ and decompose it as
follows
   \begin{align*}
h = h_n \oplus g_{n+1} \oplus \is{h}{e_{n+1}} e_{n+1},
\quad n\Ge 1,
   \end{align*}
where $h_n:=\sum_{j=1}^n \is{h}{e_j} e_j$ and
$g_n=h-h_n$. Clearly, $h_n \oplus g_{n+1} \in
\hh_{1,n}$ and $\is{h}{e_{n+1}} e_{n+1}\in \hh_{2,n}$.
Note that $T=W_n \oplus \widetilde{W}_n$, where
$\widetilde{W}_n\in \ogr{\bigoplus_{j=n+1}^{\infty}
\langle e_j \rangle}$ is the unitary operator given by
   \begin{align*}
\widetilde{W}_n \bigg(\sum_{j=n+1}^{\infty} \xi_j
e_j\bigg) = \sum_{j=n+1}^{\infty} \lambda_j \xi_j e_j,
\quad \{\xi_{j+n}\}_{j=1}^{\infty}\in \ell^2.
   \end{align*}
Since $Th=W_n h_n \oplus \widetilde{W}_n g_n$, we get
   \begin{align*}
\|T_n h - Th\| & = \big\|\widetilde{V}_n g_{n+1}
\oplus \sigma \is{h}{e_{n+1}} e_{n+2} \oplus z_n
\is{h}{e_{n+1}} e_{n+1} - \widetilde{W}_n g_n\big\|
   \\
& \Le \|g_{n+1}\| + (\sigma +1) |\is{h}{e_{n+1}}| +
\|g_n\|, \quad n \Ge 1.
   \end{align*}
However, $\lim_{n\to\infty} h_n=h$ and
$\lim_{n\to\infty} \is{h}{e_n}=0$, so $Th=\lim_{n\to
\infty} T_n h$.

It remains to prove that $\{T_n^*\}_{n=1}^{\infty}$
converges strongly to $T^*$. First, observe that
$T_n^*= \big[\begin{smallmatrix} V_n^* & 0 \\
\sigma E_n^* & U_n^* \end{smallmatrix}\big]$ and
$E_n^* (h_n \oplus g_{n+1}) = \is{h}{e_{n+2}}e_{n+1}$.
Hence, we have
   \begin{align*}
\|T_n^* h - T^*h\| & = \big \|\widetilde{V}_n^*
g_{n+1} \oplus \sigma \is{h}{e_{n+2}}e_{n+1} + \bar
z_n \is{h}{e_{n+1}}e_{n+1} - \widetilde{W}_n^* g_n
\big\|
   \\
&\Le \|g_{n+1}\| + \sigma |\is{h}{e_{n+2}}| +
|\is{h}{e_{n+1}}| + \|g_n\|, \quad n \Ge 1,
   \end{align*}
which implies that $T^*h = \lim_{n\to \infty} T_n^*
h$.

Summarizing, we have proved that the sequence
$\{T_n\}_{n=1}^{\infty}$ of Brownian unitaries of
covariance $\sigma$ converges $^*$-strongly to the
unitary operator $T$.
   \hfill $\diamondsuit$
   \end{exa}
The Brownian unitaries $T_n$ appearing in
Example~\ref{prz1} resemble Brownian shifts (cf.\
\cite[Proposition~5.2]{Ag-St95-6}). In particular,
$\dim{\hh_{2,n}}=1$ for all $n$. However, if
$\mathfrak{n}$ is any cardinal number, then taking the
orthogonal sum of such operators with $\mathfrak{n}$
summands and using Proposition~\ref{orhsy}, we can
obtain a sequence of Brownian unitaries of constant
positive covariance, which $^*$-strongly converges to
a unitary operator and has the property that
$\dim{\hh_{2,n}}=\mathfrak{n}$ for all $n$. The next
example is not an orthogonal sum as above, and
$\dim{\hh_{2,n}} = \aleph_0$ for all $n$.
   \begin{exa} \label{przew2}
In this example, given an isometry $V$ on a Hilbert
space $\kk$, we write $\widehat{V}$ for the operator
$V$ regarded as an operator from $\kk$ onto $\ob{V}$.
Clearly, $\widehat{V}$ is a unitary isomorphism. Let
$\sigma$ be a fixed positive real number and
$T_0=\big[\begin{smallmatrix} V & \sigma  E \\
0 & U\end{smallmatrix}\big]$ be a Brownian unitary of
covariance $\sigma$ relative to an orthogonal
decomposition $\hh=\hh_1\oplus \hh_2$ with $\dim
\hh_1=\dim \hh_2=\aleph_0$. Assume that
$\{\mcal_n\}_{n=1}^{\infty}$ is a sequence of closed
subspaces of $\hh_2$ reducing $U$ and such that
$\hh_2= \bigoplus_{n=1}^{\infty} \mcal_n$ and $\dim
\mcal_n = \aleph_0$ for all integers $n\Ge 1$. Since
$\dim \mcal_j = \dim E(\mcal_j)\oplus \mcal_j$, there
exists a unitary isomorphism $W_j$ from $\mcal_j$ onto
$E(\mcal_j)\oplus \mcal_j$. Then $W^{(n)}:=W_1 \oplus
\ldots \oplus W_n$ is a unitary isomorphism from
$\bigoplus_{j=1}^{n} \mcal_j$ onto
$\bigoplus_{j=1}^{n} E(\mcal_j) \oplus
\bigoplus_{j=1}^{n} \mcal_j$ (use the fact that $E$ is
isometric). For $n\Ge 1$, we define $\hh_{1,n}=\hh_1
\oplus \bigoplus_{j=1}^{n} \mcal_j $ and
$\hh_{2,n}=\bigoplus_{j=n+1}^{\infty} \mcal_j$.
Clearly, $\hh=\hh_{1,n} \oplus \hh_{2,n}$ and
$\dim{\hh_{1,n}} = \dim{\hh_{2,n}} = \aleph_0$. Using
the fact that $\ob{V}\perp \ob{E}$, we can define the
isometry $V_n\in \ogr{\hh_{1,n}}$ as follows
   \begin{align*}
V_n (h_1 \oplus \hat h_n) = Vh_1 \oplus W^{(n)} \hat
h_n, \quad h_1\in \hh_1, \, \hat h_n \in
\bigoplus_{j=1}^{n} \mcal_j.
   \end{align*}
Set $E_n=E|_{\hh_{2,n}}$ and $U_n=U|_{\hh_{2,n}}$.
Then $E_n$ is an isometry from $\hh_{2,n}$ into
$\hh_{1,n}$ and $U_n$ is a unitary operator on
$\hh_{2,n}$. Since
   \begin{align} \label{werzu}
\ob{V_n}=\ob{V} \oplus \bigoplus_{j=1}^{n} E(\mcal_j)
\oplus \bigoplus_{j=1}^{n} \mcal_j, \quad n\Ge 1,
   \end{align}
and $\ob{E_n}= \bigoplus_{j=n+1}^{\infty} E(\mcal_j)$,
we infer from $\hh_1=\ob{V}\oplus \ob{E}$ that
   \begin{align*}
\ob{V_n}\oplus \ob{E_n} = \ob{V} \oplus \ob{E} \oplus
\bigoplus_{j=1}^{n} \mcal_j = \hh_{1,n}, \quad n\Ge 1.
   \end{align*}
This means that
the operator $T_n\in \ogr{\hh}$ given by $T_n=\big[\begin{smallmatrix} V_n & \sigma E_n \\
0 & U_n \end{smallmatrix}\big]$ relative to
$\hh=\hh_{1,n} \oplus \hh_{2,n}$ is a Brownian unitary
of covariance $\sigma$. Since $\hh=\ob{V} \oplus
\ob{\bigoplus_{j=1}^{\infty} W_j}$, we can define the
unitary operator $T\in \ogr{\hh}$ by
$T=\widehat{V}\oplus \bigoplus_{j=1}^{\infty} W_j$.

First, we show that $\{T_n\}_{n=1}^{\infty}$ converges
strongly to $T$. Fix a vector $h\in \hh$ and decompose
it as $h=h_{1,n} \oplus h_{2,n}$, where $h_{1,n} \in
\hh_{1,n}$ and $h_{2,n}\in \hh_{2,n}$. Noting that
$T=\widehat{V_n} \oplus \bigoplus_{j=n+1}^{\infty}
W_j$, we have
   \allowdisplaybreaks
   \begin{align*}
\|T_n h - Th\| & = \Big\|V_nh_{1,n} \oplus \sigma E_n
h_{2,n} \oplus U_n h_{2,n} - \Big[V_n h_{1,n} \oplus
\Big(\bigoplus_{j=n+1}^{\infty} W_j\Big)
h_{2,n}\Big]\Big\|
   \\
& = \Big\|\sigma E_n h_{2,n} + \Big(U_n -
\bigoplus_{j=n+1}^{\infty} W_j\Big)h_{2,n} \Big\| \Le
(\sigma + 2) \|h_{2,n}\|, \quad n\Ge 1.
   \end{align*}
Since $h_{2,n}$ converges to $0$ as $n\to \infty$,
this implies that $Th=\lim_{n\to \infty} T_n h$.

Next, we prove that $\{T_n^*\}_{n=1}^{\infty}$
converges strongly to $T^*$. Decompose a vector $h\in
\hh$ as $h=h_1 \oplus g_{n} \oplus h_{2,n}$, where
$h_1\in \hh_1$, $g_n \in \bigoplus_{j=1}^{n} \mcal_j$
and $h_{2,n}\in \hh_{2,n}$. Then $h_{1,n}:=h_1 \oplus
g_n\in \hh_{1,n}$. Denote by $Q_n$ (resp., $P_n$) the
orthogonal projection of $\hh$ (resp., $\hh_{1,n}$)
onto $\ob{V_n}$ and set $Q_n^{\bot}:=I-Q_n$. Since
$V_n^*= \widehat{V_n}^*P_n$ and $Q_n h = Q_n
h_{1,n}=P_nh_{1,n}$, we have
   \begin{align} \label{vnwj}
V_n^* h_{1,n} = \widehat{V_n}^* P_n h_{1,n} =
\widehat{V_n}^* Q_n h, \quad n \Ge 1.
   \end{align}
Noting that $T_n^*= \big[\begin{smallmatrix} V_n^* & 0 \\
\sigma E_n^* & U_n^*
\end{smallmatrix}\big]$ and  $T^*=\widehat{V_n}^* \oplus
\bigoplus_{j=n+1}^{\infty} W_j^*$, we get
   \allowdisplaybreaks
   \begin{align} \notag
\|T_n^* h - T^* h\| & = \Big\|V_n^* h_{1,n} \oplus
(\sigma E_n^* h_{1,n} + U_n^* h_{2,n})
   \\ \notag
&\hspace{10ex}- \Big[\widehat{V_n}^* Q_n h \oplus
\Big(\bigoplus_{j=n+1}^{\infty} W_j^*\Big)Q_n^{\perp}h
\Big]\Big\|
   \\  \notag
& \hspace{-1ex} \overset{\eqref{vnwj}}= \Big\|\sigma
E_n^* h_{1,n} + U_n^* h_{2,n} -
\Big(\bigoplus_{j=n+1}^{\infty} W_j^*\Big)Q_n^{\perp}h
\Big\|
   \\ \label{vaosr}
& \Le \sigma \|E_n^* h_{1,n}\| + \|h_{2,n}\| +
\|Q_n^{\perp}h\|, \quad n \Ge 1.
   \end{align}
Since $E_n$ is an isometry, $E_nE_n^*$ is the
orthogonal projection of $\hh_{1,n}$ onto $\ob{E_n}$.
Observing that $h_{1,n} - h_1 = g_n \perp \ob{E_n}$,
we see that
   \begin{align*}
E_nE_n^*h_{1,n}=E_nE_n^*h_1= R_n h_1, \quad n\Ge 1,
   \end{align*}
where $R_n$ is the orthogonal projection of $\hh_1$
onto $\ob{E_n}$. Recalling that $\ob{E_n}=
\bigoplus_{j=n+1}^{\infty} E(\mcal_j)$, we deduce that
$\{R_n\}_{n=1}^{\infty}$ converges strongly to $0$. As
a consequence, $\lim_{n\to \infty} E_nE_n^*h_{1,n}=0$,
which implies that $\lim_{n\to \infty}
E_n^*h_{1,n}=0$. Since $\hh_1=\ob{V}\oplus \ob{E}$ and
$\ob{Q_n}= \ob{V_n}$, we infer from \eqref{werzu} that
$\{Q_n\}_{n=1}^{\infty}$ converges strongly to $I$.
Hence, $\lim_{n\to \infty} Q_n^{\perp}h=0$. Knowing
that $h_{2,n}$ converges to $0$ as $n\to \infty$ and
using \eqref{vaosr}, we conclude that
$\{T_n^*\}_{n=1}^{\infty}$ converges strongly to
$T^*$.

Thus, we have proved that the sequence
$\{T_n\}_{n=1}^{\infty}$ of Brownian unitaries of
covariance $\sigma$ converges $^*$-strongly to the
unitary operator $T$.
   \hfill $\diamondsuit$
   \end{exa}


   \end{document}